\documentclass[11pt,reqno]{amsart}

\usepackage[T1]{fontenc}
\usepackage{lmodern}
\usepackage{microtype}
\usepackage[margin=1.03in]{geometry}
\usepackage{amsmath,amssymb,amsfonts,mathtools}
\usepackage{amsthm}
\usepackage{booktabs}
\usepackage{enumitem}
\usepackage{needspace}
\usepackage{xcolor}
\usepackage[colorlinks=true,
  linkcolor=blue!50!black,
  citecolor=blue!50!black,
  urlcolor=blue!50!black]{hyperref}
\usepackage[capitalize,noabbrev]{cleveref}

\allowdisplaybreaks
\numberwithin{equation}{section}
\setlist[enumerate]{leftmargin=2.2em,itemsep=0.25em,topsep=0.4em}

\theoremstyle{plain}
\newtheorem{theorem}{Theorem}[section]
\newtheorem{proposition}[theorem]{Proposition}
\newtheorem{lemma}[theorem]{Lemma}

\theoremstyle{remark}
\newtheorem{remark}[theorem]{Remark}

\newcommand{\Z}{\mathbb Z}
\newcommand{\Q}{\mathbb Q}
\newcommand{\qf}{\operatorname{qf}}
\newcommand{\ac}{\operatorname{ac}}
\newcommand{\Spec}{\operatorname{Spec}}

\title[Alternating Divisoriality]{Alternating Divisoriality of Prime Ideals along \\ Infinite Chains of B\'ezout Domains}

\author{Viet-Hoang Tran}
\address{Department of Mathematics, National University of Singapore, Singapore 119076}
\email{hoang.tranviet@u.nus.edu}
\urladdr{https://vh-tran.github.io/}

\author{Tan M. Nguyen}
\address{Department of Mathematics, National University of Singapore, Singapore 119076}
\email{tanmn@nus.edu.sg}
\urladdr{https://tanmnguyen89.github.io/}

\date{}


\hypersetup{
  pdftitle={Alternating Divisoriality of Prime Ideals along Infinite Chains of Bezout Domains},
  pdfauthor={Viet-Hoang Tran}
}

\begin{document}
\raggedbottom

\begin{abstract}
A question posed in the literature asks whether the divisoriality of a prime
can alternate indefinitely along an ascending chain of overrings of a
Pr\"ufer domain or a descending chain of same-quotient-field underrings, and
what happens at the limit.  We construct such chains with all finite-stage
domains B\'ezout and show, in both directions, that the limit prime can be
either divisorial or nondivisorial.
\end{abstract}

\maketitle

\section{Background and main results}

When a domain is enlarged inside its quotient field, or replaced by a smaller
domain with the same quotient field, ideals may be compared by extension and
contraction.  The question considered here was posed in
\cite[Problem~12]{CFFG2014}: can a prime, followed through such an infinite
chain, alternate at every stage with respect to a natural ideal-closure
property, and what happens in the limiting union or intersection?  To
formulate the question, we recall the $v$-operation.

Let $D$ be a domain with quotient field $L$.  For a nonzero fractional ideal
$I$ of $D$, set
\begin{equation}\label{eq:v-operation}
 I^{-1}=(D:I)=\{x\in L:xI\subseteq D\},
 \qquad
 I^{v_D}=(D:(D:I)).
\end{equation}
The ideal $I$ is \emph{divisorial in $D$} when $I=I^{v_D}$, that is, when it
is recovered from its ideal dual by dualizing again.  Since both colons in
\eqref{eq:v-operation} are defined with respect to $D$, divisoriality depends
on the ambient domain.

An \emph{overring} of $D$ is a domain $E$ with
$D\subseteq E\subseteq L$.  A domain $T\subseteq D$ is an \emph{underring}
of $D$, and it is a \emph{same-quotient-field underring} when $\qf(T)=L$.
If $P$ is an ideal of $D$, then $PE$ denotes its extension to $E$, while
$P\cap T$ denotes its contraction to $T$; for a prime $P$, $D_P$ is the
usual localization.  A Pr\"ufer domain is a domain in which every nonzero
finitely generated ideal is invertible, while a B\'ezout domain is one in
which every finitely generated ideal is principal.  Hence every B\'ezout
domain is Pr\"ufer, every overring of a Pr\"ufer domain is Pr\"ufer, and a
nonzero nondivisorial prime of a Pr\"ufer domain is necessarily not finitely
generated.

Divisorial prime ideals in Pr\"ufer domains have been studied through ideal
duals, powers, and structural classifications
\cite{HP1982,FHP1984,FHP1986,FHL2010}.  The dependence on the ambient domain
is already visible in comparable three-ring examples exhibiting both patterns
\[
 \text{divisorial--nondivisorial--divisorial}
 \qquad\text{and}\qquad
 \text{nondivisorial--divisorial--nondivisorial};
\]
the first arises from \cite[Example~4.10]{FHL2010} after passing to the
appropriate localization.  These examples show that divisoriality need not
vary monotonically, but they stop after three stages.  We now state the
question precisely.

\begin{enumerate}[label=\textup{(\alph*)}]
 \item Do there exist a Pr\"ufer domain $R$, a nonzero prime ideal $P$ of
 $R$, and a countably infinite strictly ascending chain of overrings
 \begin{equation}\label{eq:ascending-question}
  R=R_0\subsetneq R_1\subsetneq R_2\subsetneq\cdots\subsetneq R_P
 \end{equation}
 such that $PR_{2n}$ is divisorial in $R_{2n}$ and $PR_{2n+1}$ is
 nondivisorial in $R_{2n+1}$ for every $n\geq0$?  If so, for
 \[
   S:=\bigcup_{m\geq0}R_m\subseteq R_P,
 \]
 is $PS$ divisorial in $S$?

 \item Do there exist a Pr\"ufer domain $R$, a nonzero prime ideal $P$ of
 $R$, and a countably infinite strictly descending chain of underrings
 \begin{equation}\label{eq:descending-question}
  R=R_0\supsetneq R_1\supsetneq R_2\supsetneq\cdots,
  \qquad T:=\bigcap_{m\geq0}R_m,
 \end{equation}
 with $\qf(T)=\qf(R)$, such that $P\cap R_{2n}$ is divisorial in $R_{2n}$
 and $P\cap R_{2n+1}$ is nondivisorial in $R_{2n+1}$ for every $n\geq0$?
 If so, is $P\cap T$ divisorial in $T$?
\end{enumerate}

Later criteria governing extension and contraction
\cite{Lucas2015,Lucas2017} do not produce a compatible infinite chain or
determine its limit.  To the best of our knowledge, neither existence
question nor either limit question has been resolved.  Our answer has two
parts.  First, one strictly increasing bi-infinite chain yields the ascending
example from its nonnegative half and the descending example from its
nonpositive half read backwards;
divisoriality alternates at every finite stage, and both limit primes are
nondivisorial.  Second, separate constructions make the limit prime
divisorial in either direction without changing the finite-stage alternation.
Thus the alternating hypotheses do not determine the limit behavior.  Every
finite stage in these examples is B\'ezout.

\begin{theorem}[Bilateral construction]\label{thm:main}
There exist a field $K$, a rank-one nondiscrete valuation domain $A$ of $K$
with maximal ideal $M$, and B\'ezout domains $S_j\subsetneq A$, $j\in\Z$,
each having quotient field $K$, with the following properties.
\begin{enumerate}[label=\textup{(\roman*)}]
 \item\label{item:strict} $S_j\subsetneq S_{j+1}$ for every $j\in\Z$.
 \item\label{item:primecompat} If $Q_j=M\cap S_j$, then $Q_j$ is a nonzero
 prime and, whenever $j\leq\ell$,
 \[
   Q_\ell\cap S_j=Q_j,
   \qquad Q_jS_\ell=Q_\ell.
 \]
 \item\label{item:localization} $(S_0)_{Q_0}=A$.
 \item\label{item:parity} The prime $Q_j$ is divisorial in $S_j$ if and
 only if $j$ is even.
 \item\label{item:limits} The rings
 \[
   S^+=\bigcup_{j\geq0}S_j,
   \qquad S^-=\bigcap_{j\leq0}S_j
 \]
 are B\'ezout domains with quotient field $K$, and $S^+\subsetneq A$.
 Moreover,
 \[
   M\cap S^+=Q_0S^+,
   \qquad M\cap S^-=Q_0\cap S^-,
 \]
 and both displayed primes are nondivisorial in their respective rings.
\end{enumerate}
\end{theorem}

For the ascending question, take $R=S_0$, $P=Q_0$, and $R_m=S_m$ for
$m\geq0$.  For the descending question, take the same initial pair and set
$R_m=S_{-m}$.
The compatibility and parity assertions give the required alternation, and
the limit assertion gives a nondivisorial prime in both cases.

\begin{theorem}[Both limit alternatives]\label{thm:limit-alternatives}
For both the ascending and descending questions, there are chains with a
divisorial limit prime and chains with a nondivisorial limit prime.  In the
ascending case, the opposite limit behaviors may occur with the same initial
pair $(R,P)$ and the same localization $R_P$.  Every finite stage may be
chosen to be B\'ezout.
\end{theorem}

All constructions are explicit intersections of valuation domains.  A
rank-one valuation supplies the tracked prime, auxiliary valuations control
the two colons in \eqref{eq:v-operation} and produce the parity rule, and
residue-field modifications change the limit behavior without changing that
rule.

The proof has three parts.  In
\cref{sec:preliminaries,sec:valuations,sec:chain,sec:divisoriality} we develop
the valuation-intersection machinery, construct the bilateral chain, and
prove its parity rule.  Sections~\ref{sec:positive-limit}
and~\ref{sec:negative-limit} compute the two nondivisorial limits.  Finally,
\cref{sec:divisorial-positive-limit,sec:large-avoidance,%
sec:divisorial-negative-limit} construct the two divisorial limits and
complete the proof of \cref{thm:limit-alternatives}.

\section{Preliminaries}\label{sec:preliminaries}

We begin with a valuation-intersection lemma.  Its elementary approximation
argument is also used later to identify a localization.

\begin{lemma}[Valuation intersections]\label{lem:intersection}
Let $k$ be an infinite field of cardinality $\kappa$, let $L/k$ be a field
extension, and let $(V_i)_{i\in I}$ be valuation rings of $L$, trivial on
$k$, with $|I|<\kappa$.  Then
\[
 D=\bigcap_{i\in I}V_i
\]
is a B\'ezout domain with quotient field $L$.
\end{lemma}

\begin{proof}
The assertion is immediate if the family is empty, so assume otherwise and
write $v_i$ for the valuation belonging to $V_i$.  Fix $z\in L^\times$.
For a fixed $i$ with $v_i(z)=0$, there is at most one $c\in k$ for which
$v_i(c+z)>0$: triviality on $k$ makes the residue map
$k\longrightarrow\kappa(v_i)$ injective, so two distinct constants cannot
cancel the residue of $z$.  Fewer than $\kappa$ constants are therefore
excluded by all the valuations.  After also excluding $0$ and, if $z\in k$,
$-z$, choose $c\in k^\times$ such that $c+z\neq0$ and
\begin{equation}\label{eq:avoid-zero}
 v_i(c+z)=0\quad\text{whenever }v_i(z)=0.
\end{equation}
If $v_i(z)\neq0$, the two summands $c$ and $z$ have distinct values, because
$v_i(c)=0$; hence the ultrametric inequality is an equality.  Together with
\cref{eq:avoid-zero}, this gives
\begin{equation}\label{eq:min-value}
 v_i(c+z)=\min\{0,v_i(z)\}\qquad(i\in I).
\end{equation}
It follows that
\[
 \frac1{c+z},\qquad\frac z{c+z}
\]
belong to every $V_i$, and therefore to $D$.  Their quotient is $z$, so
$\qf(D)=L$.

Now let $0\neq a,b\in D$, put $z=b/a$, choose $c$ as above, and set
$d=ca+b=a(c+z)$.  For every $i$, \cref{eq:min-value} yields
\[
 v_i(d)=v_i(a)+\min\{0,v_i(b)-v_i(a)\}
       =\min\{v_i(a),v_i(b)\}.
\]
Thus $a/d,b/d\in D$, so $a,b\in dD$.  Conversely,
$d=ca+b\in aD+bD$.  Hence $aD+bD=dD$.  Induction shows that every
finitely generated ideal is principal, and $D$ is B\'ezout.
\end{proof}

We will also use the standard extension property for a prime of a Pr\"ufer
overring.  A proof is included to fix all contractions and localizations.

\begin{lemma}[Primes in Pr\"ufer overrings]\label{lem:prime-extension}
Let $D$ be a Pr\"ufer domain and
$D\subseteq E\subseteq\qf(D)$.  If $\mathfrak q\in\Spec(E)$ and
$\mathfrak p=\mathfrak q\cap D$, then
\begin{equation}\label{eq:prime-extension}
 \mathfrak q=\mathfrak pE.
\end{equation}
\end{lemma}

\begin{proof}
We first recall why every overring of a Pr\"ufer domain is flat.  Let $V$ be
a valuation domain and let $V\subseteq C\subseteq\qf(V)$.  The intermediate
ring $C$ is itself a valuation domain: for $0\neq x\in\qf(V)$, either
$x\in V\subseteq C$ or $x^{-1}\in V\subseteq C$.  Let $\mathfrak n$ be the
maximal ideal of $C$ and put $\mathfrak r=\mathfrak n\cap V$.  Then
$\mathfrak r$ is prime.  Every element of $V\setminus\mathfrak r$ is a unit
of $C$, so $V_{\mathfrak r}\subseteq C$.  Conversely, if
$x\in C\setminus V$, then $x^{-1}\in V$; because $x$ is its inverse in $C$,
$x^{-1}\notin\mathfrak n$, and hence $x^{-1}\notin\mathfrak r$.  Therefore
$x\in V_{\mathfrak r}$, proving $C=V_{\mathfrak r}$.  Thus every overring of
a valuation domain is flat.  For each maximal ideal $N$ of $D$, the ring
$D_N$ is a valuation domain and $E_{D\setminus N}$ is an overring of $D_N$.
Flatness is local, so $E$ is flat over $D$.

Set $S=D\setminus\mathfrak p$ and $V=D_{\mathfrak p}$.  By the preceding
paragraph, $S^{-1}E=V_{\mathfrak r}$ for some prime
$\mathfrak r\subseteq\mathfrak pV$.  The prime
$S^{-1}\mathfrak q$ survives and its contraction to $V$ is
$\mathfrak pV$.  Every element of $V\setminus\mathfrak pV$ is already
inverted in $S^{-1}E$, so $\mathfrak r\subseteq\mathfrak pV$; on the other
hand, the contraction to $V$ of any prime of $V_{\mathfrak r}$ is contained
in $\mathfrak r$.  Hence $\mathfrak pV\subseteq\mathfrak r$, and consequently
\[
 \mathfrak r=\mathfrak pV,\qquad S^{-1}E=V,
 \qquad S^{-1}\mathfrak q=\mathfrak pV.
\]
Moreover,
$S^{-1}(\mathfrak pE)=\mathfrak p(S^{-1}E)=\mathfrak pV$.
If $x\in\mathfrak q$, there is therefore an $s\in S$ such that
$sx\in\mathfrak pE$.  Since $E$ is flat over $D$, the
$D/\mathfrak p$-module $E/\mathfrak pE$ is flat and hence torsion-free.
The class of $s$ in the domain $D/\mathfrak p$ is nonzero, so
$s(x+\mathfrak pE)=0$ forces $x\in\mathfrak pE$.  This proves
$\mathfrak q\subseteq\mathfrak pE$; the reverse inclusion follows from
$\mathfrak p\subseteq\mathfrak q$.
\end{proof}

\section{The field and the valuation families}\label{sec:valuations}

Fix an uncountable field $k$ of characteristic zero.  Let
\[
 \Lambda=\Z\sqcup\{\ast\},
 \qquad \Gamma=\bigoplus_{i\in\Lambda}\Q e_i.
\]
The group $\Gamma$ admits a translation-invariant total order: enumerate the
countable set $\Lambda$ and order two elements lexicographically at their
first unequal coordinate.  With respect to this order, the least term of a
product in the group algebra $k[\Gamma]$ is the product of the two least
terms.  It follows that $k[\Gamma]$ is a domain.  Write $X_i^r$ for the basis
element corresponding to $re_i$, and put
\begin{equation}\label{eq:fields}
 F=\operatorname{Frac}(k[\Gamma]),
 \qquad K=\operatorname{Frac}(F[T^{\Q}]).
\end{equation}
The same least-term argument shows directly that $F[T^{\Q}]$ is a domain.
Every particular element of $F$ involves only finitely many of the variables
$X_i$, and every element of $K$ involves only finitely many rational powers
of $T$.

For $i\in\Lambda$, let
\[
 u_i:F^\times\longrightarrow\Q
\]
be the monomial valuation determined by
\begin{equation}\label{eq:coordinate-valuation}
 u_i(X_i^r)=r,
 \qquad u_i(X_j^r)=0\quad(j\neq i).
\end{equation}
More explicitly, group the finitely many terms of a nonzero element of
$k[\Gamma]$ according to their $i$th exponent and take the least exponent
whose homogeneous part is nonzero.  The product of the two least homogeneous
parts is nonzero, since the group algebra on
$\bigoplus_{j\neq i}\Q e_j$ is a domain.  Thus this rule is multiplicative,
and it extends uniquely to the stated valuation on $F$.

Let
\[
 v:K^\times\longrightarrow\Q
\]
be the analogous $T$-exponent valuation.  Its valuation ring and maximal
ideal are
\begin{equation}\label{eq:A-M}
 A=\{z\in K:v(z)\geq0\},
 \qquad M=\{z\in K:v(z)>0\}.
\end{equation}
The value group is $\Q$ and the residue field is $F$.  Hence $A$ is a
rank-one nondiscrete valuation domain.  We extend every valuation to zero by
the value $+\infty$.

For $0\neq z\in K$, set $q=v(z)$.  The image of $T^{-q}z$ in the residue
field $F$ is the \emph{angular component} of $z$, denoted $\ac(z)$.  For each
$i\in\Lambda$, composition of $v$ with the residue valuation $u_i$ gives the
valuation
\[
 \beta_i(z)=\bigl(v(z),u_i(\ac(z))\bigr)
       \in\Q\mathbin{\mathrm{lex}}\Q.
\]
Its valuation ring is
\begin{equation}\label{eq:B-i}
 B_i=\{z:v(z)>0\}\ \cup\
     \{z:v(z)=0,\ u_i(\ac(z))\geq0\}.
\end{equation}
In particular, $B_i\subseteq A$.

We next construct valuation rings which asymptotically force membership in
$B_i$.  For $i\in\Lambda$ and $m\geq1$, begin with the assignments
\[
 T^r\longmapsto X_i^{mr}(1+Y)^r\quad(r\in\Q),
\]
where
\[
 (1+Y)^r=\sum_{h\geq0}\binom rhY^h\in k[[Y]].
\]
The coefficients exist because $k$ has characteristic zero, and the formal
binomial identity
$(1+Y)^r(1+Y)^s=(1+Y)^{r+s}$ makes the assignments compatible with
addition in $\Q$.  They therefore define an $F$-algebra homomorphism
$\sigma^0_{i,m}:F[T^{\Q}]\to F[[Y]]$.

This homomorphism is injective.  Every Laurent polynomial in $F[T^{\Q}]$
belongs to $F[T^{1/N},T^{-1/N}]$ for some $N\geq1$, and $T^{1/N}$ maps to
\[
 w=X_i^{m/N}(1+Y)^{1/N}.
\]
If $w$ were algebraic over $F$, then
\[
 Y=\left(w/X_i^{m/N}\right)^N-1
\]
would be algebraic over $F$, contradicting the transcendence of $Y$.
Thus $w$ is transcendental over $F$, so the restriction of $\sigma^0_{i,m}$
to every such Laurent polynomial ring is injective.  Hence
$\sigma^0_{i,m}$ is injective and extends uniquely to an embedding
\begin{equation}\label{eq:sigma}
 \sigma_{i,m}:K\longrightarrow F((Y)),
 \qquad T^r\longmapsto X_i^{mr}(1+Y)^r.
\end{equation}
We will also use the resulting identity
\begin{equation}\label{eq:K-union}
 K=\bigcup_{N\geq1}F(T^{1/N}).
\end{equation}

Let $d_{i,m}=\operatorname{ord}_Y\circ\sigma_{i,m}$.  This is a discrete
valuation on $K$ with residue field $F$: the residue of a zero-valued element
is its nonzero constant term, and constant elements realize all of $F$.
Compose $d_{i,m}$ with $u_i$ on this residue field.  We denote the resulting
valuation ring by
\begin{equation}\label{eq:W-i-m}
 W_{i,m}=\{z:d_{i,m}(z)>0\}\ \cup\
 \{z:d_{i,m}(z)=0,\ u_i(\overline z)\geq0\},
\end{equation}
where $\overline z\in F^\times$ is the residue of $z$.  The element
$\pi_{i,m}=TX_i^{-m}-1$ maps to $Y$, while constants with arbitrary
$u_i$-value have $d_{i,m}$-value zero.  Thus the composite has value group
$\Z\mathbin{\mathrm{lex}}\Q$.  If $z\in F(T^{1/N})$, write $z(X_i^m)$ for
the result of evaluating $T^{1/N}$ at $X_i^{m/N}$ whenever a chosen
denominator is nonzero.  This value is independent of $N$ and of the
rational presentation, because the rational powers are compatible and
evaluation is a homomorphism on its local domain.  If $z$ has neither a zero
nor a pole under this evaluation, then $d_{i,m}(z)=0$ and
\begin{equation}\label{eq:ordinary-specialization}
 \overline z=z(X_i^m).
\end{equation}
All valuations defining $A$, $B_i$, and $W_{i,m}$ are trivial on $k$.

\begin{lemma}[Limiting implication]\label{lem:limiting}
For every $i\in\Lambda$,
\begin{equation}\label{eq:limiting-inclusion}
 A\cap\bigcap_{m\geq1}W_{i,m}\subseteq B_i.
\end{equation}
\end{lemma}

\begin{proof}
Fix $0\neq z\in K$.  Choose $N\geq1$ such that $z\in F(Z)$, where
$Z=T^{1/N}$, and write
\[
 z=\frac{f(Z)}{g(Z)},\qquad
 f(Z)=\sum_{\ell=r_0}^{r_1}a_\ell Z^\ell,
 \qquad g(Z)=\sum_{\ell=s_0}^{s_1}b_\ell Z^\ell,
\]
with $a_{r_0}b_{s_0}\neq0$.  Then
\begin{equation}\label{eq:v-ac-polynomials}
 v(z)=\frac{r_0-s_0}{N},
 \qquad \ac(z)=\frac{a_{r_0}}{b_{s_0}}.
\end{equation}
For every $\ell>r_0$ with $a_\ell\neq0$,
\begin{align*}
 &u_i\bigl(a_\ell X_i^{m\ell/N}\bigr)
   -u_i\bigl(a_{r_0}X_i^{mr_0/N}\bigr)\\
 &\hspace{4em}=u_i(a_\ell)-u_i(a_{r_0})
       +\frac{m(\ell-r_0)}N,
\end{align*}
which is positive for all sufficiently large $m$.  Hence the $r_0$-term is
the unique term of least $u_i$-value in $f(X_i^{m/N})$ for all large $m$;
in particular, this evaluation is nonzero.  The same argument applies to
$g$ and its $s_0$-term.  Therefore $z$ has neither a zero nor a pole at
$T=X_i^m$ for all sufficiently large $m$, and
\begin{align}
 u_i\bigl(z(X_i^m)\bigr)
 &=u_i(a_{r_0})+\frac{mr_0}N
   -u_i(b_{s_0})-\frac{ms_0}N \notag\\
 &=m v(z)+u_i(\ac(z)).\label{eq:asymptotic-value}
\end{align}

Now suppose $z$ belongs to the left side of
\cref{eq:limiting-inclusion}.  If $v(z)>0$, then $z\in B_i$ by
\cref{eq:B-i}.  If $v(z)=0$, membership in $W_{i,m}$, together with
\cref{eq:ordinary-specialization,eq:asymptotic-value} for large $m$, gives
$u_i(\ac(z))\geq0$.  Again $z\in B_i$ by \cref{eq:B-i}.  The case $z=0$ is
automatic.
\end{proof}

\section{The bilateral chain}\label{sec:chain}

Put
\[
 C_\ast=\bigcap_{m\geq1}W_{\ast,m}.
\]
For every $n\in\Z$, define
\begin{align}
 H_n&=A\cap C_\ast\cap
       \bigcap_{\substack{i\in\Z,\ i>n\\m\geq1}}W_{i,m},
       \label{eq:H-n}\\
 E_n&=B_n\cap C_\ast\cap
       \bigcap_{\substack{i\in\Z,\ i>n\\m\geq1}}W_{i,m},
       \label{eq:E-n}\\
 S_{2n}&=E_n,\qquad S_{2n+1}=H_n.\label{eq:S-j}
\end{align}
Each ring in \cref{eq:S-j} is the intersection of a countable family of
valuation rings of $K$ whose valuations are trivial on the uncountable field
$k$.  By \cref{lem:intersection}, every $S_j$ is a B\'ezout domain with
quotient field $K$, and hence is a Pr\"ufer domain.

\begin{proposition}\label{prop:strict-chain}
For every $n\in\Z$,
\[
 E_n\subsetneq H_n\subsetneq E_{n+1}.
\]
Consequently, $(S_j)_{j\in\Z}$ is a strictly increasing bi-infinite chain.
Every $S_j$ is a proper subring of $A$.
\end{proposition}

\begin{proof}
Since $B_n\subseteq A$, the definitions give $E_n\subseteq H_n$.  The
element $X_n^{-1}$ belongs to $A$.  It has value zero under $u_\ast$ and
under every $u_i$ with $i>n$, so it is a unit in every valuation ring
occurring in \cref{eq:H-n}.  Hence
\begin{equation}\label{eq:first-witness}
 X_n^{-1}\in H_n.
\end{equation}
On the other hand, $v(X_n^{-1})=0$ and
$u_n(\ac(X_n^{-1}))=-1$, so $X_n^{-1}\notin B_n$.  Thus
$E_n\subsetneq H_n$.

The definition of $H_n$ includes every $W_{n+1,m}$.  By
\cref{lem:limiting}, $H_n\subseteq B_{n+1}$.  Removing the family
$(W_{n+1,m})_{m\geq1}$ leaves precisely the remaining conditions in
\cref{eq:E-n} for $E_{n+1}$; hence $H_n\subseteq E_{n+1}$.

Set $x=TX_{n+1}^{-2}$.  Since $v(x)=1$, we have $x\in B_{n+1}$.  At
$W_{\ast,m}$ its specialization has $u_\ast$-value $m>0$, and at
$W_{i,m}$ with $i>n+1$ its specialization has $u_i$-value $m>0$.
Thus $x\in E_{n+1}$.  At $W_{n+1,1}$, however, it has zero first component
and residue $X_{n+1}^{-1}$, whose $u_{n+1}$-value is $-1$.  Therefore
\begin{equation}\label{eq:second-witness}
 TX_{n+1}^{-2}\in E_{n+1}\setminus H_n,
\end{equation}
and $H_n\subsetneq E_{n+1}$.

Finally, $X_{n+1}^{-1}\in A$ but it is not in $W_{n+1,m}$ for any $m$.
Hence it does not belong to $H_n$, and therefore not to $E_n\subseteq H_n$.
Thus both $H_n$ and $E_n$ are proper subrings of $A$.
\end{proof}

Set
\begin{equation}\label{eq:P}
 P=M\cap S_0=M\cap E_0.
\end{equation}
This is a prime ideal, being the contraction of the prime $M$ of $A$.
It is nonzero because $T\in P$.

\begin{proposition}\label{prop:localization}
The localization of $S_0$ at $P$ is $A$.
\end{proposition}

\begin{proof}
Because $S_0\subseteq A$ and every element of $S_0\setminus P$ has
$v$-value zero, $(S_0)_P\subseteq A$.

Conversely, fix $z\in A$.  Apply the avoidance argument in the proof of
\cref{lem:intersection} simultaneously to $v$ and to the countable family
of valuations whose rings are
\begin{equation}\label{eq:S0-family}
 B_0,\qquad W_{\ast,m}\ (m\geq1),
 \qquad W_{i,m}\ (i>0,\ m\geq1).
\end{equation}
Thus we may choose $c\in k^\times$ with $c+z\neq0$ such that, for every
valuation $w$ in this augmented family,
\[
 w(c+z)=\min\{0,w(z)\}.
\]
Consequently
\[
 a=\frac1{c+z},\qquad b=\frac z{c+z}
\]
have nonnegative value for every valuation in \cref{eq:S0-family}; hence
$a,b\in S_0$.  Since $z\in A$, the additional condition for $v$ gives
$v(c+z)=0$, and therefore $v(a)=0$.  Thus $a\notin P$, while $z=b/a$.
It follows that $z\in(S_0)_P$.
\end{proof}

For every $j\in\Z$, define
\begin{equation}\label{eq:Q-j}
 Q_j=M\cap S_j.
\end{equation}
Each $Q_j$ is prime.  It is nonzero, since $T\in Q_j$: this also follows
from the stronger calculation in \cref{eq:T-e-universal} below.

\begin{proposition}[Compatibility of the primes]\label{prop:compatibility}
If $j\leq\ell$, then
\[
 Q_\ell\cap S_j=Q_j,
 \qquad Q_jS_\ell=Q_\ell.
\]
\end{proposition}

\begin{proof}
The first equality is immediate from $S_j\subseteq S_\ell$ and
\cref{eq:Q-j}:
\[
 Q_\ell\cap S_j=(M\cap S_\ell)\cap S_j=M\cap S_j=Q_j.
\]
Apply \cref{lem:prime-extension} to the Pr\"ufer domain $S_j$, its overring
$S_\ell$, and the prime $Q_\ell$.  Its contraction is $Q_j$, so the lemma
gives $Q_\ell=Q_jS_\ell$.
\end{proof}

\section{The divisoriality calculation}\label{sec:divisoriality}

We first construct elements that will test all the ideal quotients.  For every
$e\in\Q$ with $e>0$,
\begin{equation}\label{eq:T-e-universal}
 T^e\in M\cap S_j\qquad(j\in\Z).
\end{equation}
Indeed, $v(T^e)=e>0$, so $T^e\in A$ and $T^e\in B_i$ for every $i$.
At $W_{i,m}$ its first value is zero and its residue is $X_i^{me}$, of
$u_i$-value $me>0$.  This proves \cref{eq:T-e-universal} from the defining
intersections.

For $i\in\Lambda$ and $m\geq1$, set
\begin{equation}\label{eq:p-i-m}
 p_{i,m}=TX_i^{-m}\prod_{s=1}^{m-1}(1-TX_i^{-s}),
\end{equation}
where the empty product for $m=1$ is $1$.  Since each factor in the product
has $v$-value zero and angular component $1$, we have $v(p_{i,m})=1$.
Consequently $p_{i,m}\in M\cap B_j$ for every $j\in\Lambda$.

We next check every auxiliary valuation ring.  Under $\sigma_{i,t}$, the
factor with $s=t$ is $-Y$ when $t<m$.  All other factors have nonzero
constant term.  When $t\geq m$, there is no vanishing factor, and direct
specialization gives the residue values shown below.  For $j\neq i$, none of
the factors vanishes and $u_j$ sees only the specialized $T$-power.  Thus
\begin{center}
\begin{tabular}{@{}lll@{}}
\toprule
valuation ring & first component & second component if the first is $0$\\
\midrule
$W_{i,t}$, $t<m$ & $1$ & ---\\
$W_{i,m}$ & $0$ & $0$\\
$W_{i,t}$, $t>m$ & $0$ & $t-m>0$\\
$W_{j,t}$, $j\neq i$ & $0$ & $t>0$\\
\bottomrule
\end{tabular}
\end{center}
For completeness, in the second and third rows the nonvanishing factors
$1-X_i^{t-s}$ have $u_i$-value zero; the prefactor has $u_i$-value $t-m$.
In the fourth row the factors $1-X_j^tX_i^{-s}$ have $u_j$-value zero and
the prefactor $X_j^tX_i^{-m}$ has $u_j$-value $t$.  Therefore
\begin{equation}\label{eq:p-properties}
 p_{i,m}\in M\cap
   \bigcap_{\substack{j\in\Lambda\\t\geq1}}W_{j,t},
 \qquad p_{i,m}\in W_{i,m}^{\times}.
\end{equation}

We repeatedly use two elementary facts.  If $D=\bigcap_\lambda V_\lambda$
inside a common quotient field and $I$ is a nonzero $D$-submodule of that
field, then
\begin{equation}\label{eq:colon-intersection}
 (D:I)=\bigcap_\lambda(V_\lambda:I),
\end{equation}
because both sides say that $xI$ lies in every $V_\lambda$.  Also, if
$I\subseteq V$ and $I$ contains a unit of the valuation ring $V$, then
\begin{equation}\label{eq:unit-colon}
 (V:I)=V.
\end{equation}
The inclusion $V\subseteq(V:I)$ is clear; if $p\in I\cap V^\times$ and
$xI\subseteq V$, then $xp\in V$ and hence $x\in V$.

\begin{proposition}[Odd stages]\label{prop:odd}
For every $n\in\Z$, the prime
$Q_{2n+1}=M\cap H_n$ is nondivisorial in $H_n$.
\end{proposition}

\begin{proof}
Write $Q=Q_{2n+1}$.  We first claim
\begin{equation}\label{eq:A-colon-odd}
 (A:Q)=A.
\end{equation}
Certainly $AQ\subseteq A$.  Conversely, if $x\notin A$, then $v(x)<0$.
Choose $e\in\Q$ with $0<e<-v(x)$.  By
\cref{eq:T-e-universal}, $T^e\in Q$, whereas
$v(xT^e)=v(x)+e<0$.  Thus $xQ\nsubseteq A$, proving
\cref{eq:A-colon-odd}.

Consider any valuation ring $W_{i,m}$ occurring in the defining intersection
\cref{eq:H-n}; thus either $i=\ast$, or $i\in\Z$ and $i>n$.  The element
$p_{i,m}$ has positive $v$-value and belongs to every auxiliary valuation
ring by \cref{eq:p-properties}.  Hence $p_{i,m}\in Q$, while it is a unit of
$W_{i,m}$.  Equation~\eqref{eq:unit-colon} gives
\begin{equation}\label{eq:W-colon-odd}
 (W_{i,m}:Q)=W_{i,m}.
\end{equation}
Equivalently, the reverse inclusion follows explicitly from
$xQ\subseteq W_{i,m}\Rightarrow xp_{i,m}\in W_{i,m}$ and multiplication by
$p_{i,m}^{-1}$.

Taking colons through \cref{eq:H-n} and using
\cref{eq:A-colon-odd,eq:W-colon-odd}, we obtain
\begin{align}
 (H_n:Q)
 &=(A:Q)\cap\bigcap_{m\geq1}(W_{\ast,m}:Q)
  \cap\bigcap_{\substack{i\in\Z,\ i>n\\m\geq1}}(W_{i,m}:Q)\notag\\
 &=H_n.\label{eq:H-colon-Q}
\end{align}
Therefore
\[
 Q^{v_{H_n}}=(H_n:(H_n:Q))=(H_n:H_n)=H_n.
\]
Since $1\notin Q$ whereas $T\in Q$, the prime $Q$ is proper and nonzero.
In particular $Q\neq H_n$, so $Q$ is nondivisorial.
\end{proof}

\begin{proposition}[Even stages]\label{prop:even}
For every $n\in\Z$, the prime $Q_{2n}=M\cap E_n$ is divisorial in $E_n$.
\end{proposition}

\begin{proof}
Write $Q=Q_{2n}$.  We first compute
\begin{equation}\label{eq:B-colon-Q}
 (B_n:Q)=A.
\end{equation}
If $x\in A$ and $q\in Q$, then $v(q)>0$, so
$v(xq)>0$ and $xq\in B_n$ by \cref{eq:B-i}.  Hence
$A\subseteq(B_n:Q)$.  If $x\notin A$, choose
$e\in\Q$ with $0<e<-v(x)$.  Then $T^e\in Q$ by
\cref{eq:T-e-universal}, but $v(xT^e)<0$, so
$xT^e\notin B_n$.  This proves the reverse inclusion in
\cref{eq:B-colon-Q}.

For every $W_{i,m}$ occurring in \cref{eq:E-n}, the same unit argument using
$p_{i,m}\in Q$ gives $(W_{i,m}:Q)=W_{i,m}$.  Therefore
\begin{align}
 (E_n:Q)
 &=A\cap C_\ast\cap
   \bigcap_{\substack{i\in\Z,\ i>n\\m\geq1}}W_{i,m}\notag\\
 &=H_n.\label{eq:E-colon-Q}
\end{align}

It remains to calculate the conductor from $H_n$ to $E_n$.  We claim that
\begin{equation}\label{eq:E-colon-H}
 (E_n:H_n)=Q.
\end{equation}
If $q\in Q$ and $h\in H_n$, then both elements lie in every auxiliary
valuation ring defining $E_n$.  Moreover, $v(q)>0$ and $v(h)\geq0$, so
$v(qh)>0$ and therefore $qh\in B_n$.  Thus $QH_n\subseteq E_n$, proving
$Q\subseteq(E_n:H_n)$.

Conversely, suppose $xH_n\subseteq E_n$.  Since $1\in H_n$, we have
$x\in E_n$.  If $x\notin Q$, then $v(x)=0$, and membership in $B_n$ gives
\[
 a:=u_n(\ac(x))\geq0.
\]
Choose $e\in\Q$ with $e>a$.  The element $X_n^{-e}$ has $v$-value zero and
is a unit of every $W_{\ast,m}$ and every $W_{i,m}$ with $i>n$, because the
corresponding coordinate valuation vanishes on $X_n$.  Hence
$X_n^{-e}\in H_n$.  But
\[
 u_n\bigl(\ac(xX_n^{-e})\bigr)=a-e<0,
\]
so $xX_n^{-e}\notin B_n$ and hence $xX_n^{-e}\notin E_n$, a contradiction.
Thus $x\in Q$, completing the proof of \cref{eq:E-colon-H}.

Combining \cref{eq:E-colon-Q,eq:E-colon-H}, we conclude that
\[
 Q^{v_{E_n}}=(E_n:(E_n:Q))=(E_n:H_n)=Q.
\]
Thus $Q_{2n}$ is divisorial.
\end{proof}

The two propositions give the exact parity rule
\begin{equation}\label{eq:parity-rule}
 Q_j\text{ is divisorial in }S_j
 \quad\Longleftrightarrow\quad j\text{ is even}.
\end{equation}

\section{The positive union}\label{sec:positive-limit}

Define
\[
 S^+=\bigcup_{j\geq0}S_j.
\]
We first identify this ring exactly.

\begin{proposition}\label{prop:positive-union}
One has
\begin{equation}\label{eq:positive-union}
 S^+=A\cap C_\ast.
\end{equation}
Consequently $S^+$ is a B\'ezout domain with quotient field $K$, and
$S^+\subsetneq A$.
\end{proposition}

\begin{proof}
Every $S_j$ with $j\geq0$ is contained in both $A$ and $C_\ast$, so
$S^+\subseteq A\cap C_\ast$.

For the reverse inclusion, take $z\in A\cap C_\ast$.  The case $z=0$ is
immediate, so assume $z\neq0$.  By
\cref{eq:fields,eq:K-union}, there are a finite set $J\subseteq\Z$, an
integer $N\geq1$, and nonzero polynomials $f,g\in F_J[Z]$ such that
\begin{equation}\label{eq:finite-presentation}
 z=\frac{f(Z)}{g(Z)},\qquad Z=T^{1/N},
\end{equation}
where
\[
 F_J=\operatorname{Frac}\!\left(
 k\left[\bigoplus_{i\in J\cup\{\ast\}}\Q e_i\right]\right).
\]
Choose $n\geq0$ larger than every member of $J$.  If $i>n$, then $u_i$ is
zero on every nonzero coefficient of $f$ and $g$.  Write $r_0$ and $s_0$
for their respective least exponents.  After substituting
$Z=X_i^{m/N}$, the term of exponent $r_0$ is uniquely of least $u_i$-value
in $f$, and the term of exponent $s_0$ is uniquely of least value in $g$.
Thus neither evaluation vanishes, and for every $m\geq1$,
\begin{equation}\label{eq:far-coordinate-value}
 u_i\bigl(z(X_i^m)\bigr)
   =\frac{m(r_0-s_0)}N=mv(z)\geq0.
\end{equation}
It follows from \cref{eq:W-i-m,eq:ordinary-specialization} that
$z\in W_{i,m}$ for every $i>n$ and $m\geq1$.  Together with
$z\in A\cap C_\ast$, this says $z\in H_n=S_{2n+1}\subseteq S^+$, proving
\cref{eq:positive-union}.

The right side of \cref{eq:positive-union} is a countable intersection of
valuation rings trivial on $k$, so \cref{lem:intersection} makes it a
B\'ezout domain with quotient field $K$.  Finally,
$X_\ast^{-1}\in A$ but $X_\ast^{-1}\notin W_{\ast,1}$, since its
$u_\ast$-value is $-1$.  Hence $S^+\subsetneq A$.
\end{proof}

Let
\begin{equation}\label{eq:Q-plus}
 Q^+=M\cap S^+.
\end{equation}
This is a nonzero proper prime: $T\in Q^+$ and $1\notin Q^+$.

\begin{proposition}\label{prop:positive-prime}
The prime $Q^+$ satisfies $Q^+=PS^+$ and is nondivisorial in $S^+$.
\end{proposition}

\begin{proof}
Its contraction to $S_0$ is
\[
 Q^+\cap S_0=(M\cap S^+)\cap S_0=M\cap S_0=P.
\]
Therefore \cref{lem:prime-extension}, applied to $S_0\subseteq S^+$, gives
$Q^+=PS^+$.

Certainly $AQ^+\subseteq A$.  Conversely, if $x\notin A$, choose
$e\in\Q_{>0}$ with $e<-v(x)$.  Then $T^e\in Q^+$, whereas
$v(xT^e)<0$, so $xT^e\notin A$.  Hence $(A:Q^+)=A$.  For every
$m\geq1$, the element $p_{\ast,m}$ belongs to
$Q^+$ by \cref{eq:p-properties} and is a unit of $W_{\ast,m}$.  Hence
$(W_{\ast,m}:Q^+)=W_{\ast,m}$.  Taking the colon through
\cref{eq:positive-union}, we get
\[
 (S^+:Q^+)=A\cap\bigcap_{m\geq1}W_{\ast,m}=S^+.
\]
It follows that
\[
 (Q^+)^{v_{S^+}}=(S^+:(S^+:Q^+))=(S^+:S^+)=S^+\neq Q^+.
\]
Thus $Q^+$ is nondivisorial.
\end{proof}

\section{The negative intersection}\label{sec:negative-limit}

Define
\[
 S^-=\bigcap_{j\leq0}S_j=\bigcap_{m\geq0}S_{-m}.
\]

\begin{proposition}\label{prop:negative-intersection}
One has
\begin{equation}\label{eq:negative-intersection}
 S^-=A\cap C_\ast\cap
       \bigcap_{\substack{i\in\Z\\m\geq1}}W_{i,m}.
\end{equation}
Consequently $S^-$ is a B\'ezout domain with quotient field $K$.
\end{proposition}

\begin{proof}
Let $x\in S^-$.  Every $S_j$ lies in $A\cap C_\ast$, so
$x\in A\cap C_\ast$.  Fix $i\in\Z$ and $m\geq1$, and choose an integer
$n<\min\{i,0\}$.  Then $E_n=S_{2n}$ is one of the rings indexed by a
nonpositive integer in the intersection defining $S^-$.  Because $i>n$,
the valuation ring $W_{i,m}$ occurs in the defining intersection
\cref{eq:E-n} of $E_n$.  Thus $x\in W_{i,m}$.  Since $i$ and $m$ were
arbitrary, $x$ belongs to the right side of
\cref{eq:negative-intersection}.

Conversely, suppose $x$ belongs to the right side.  It belongs directly to
every $H_n$, because it satisfies $A$, $C_\ast$, and every integer-indexed
$W$-condition.  For each $n$, it also lies in
$A\cap\bigcap_{m\geq1}W_{n,m}$, so \cref{lem:limiting} gives $x\in B_n$.
It follows that $x\in E_n$ as well.  Hence $x$ lies in every $S_j$, and in
particular in every $S_j$ with $j\leq0$.  This proves the reverse inclusion.

The right side of \cref{eq:negative-intersection} is a countable
intersection of valuation rings trivial on $k$.  It is therefore a B\'ezout
domain with quotient field $K$ by \cref{lem:intersection}.
\end{proof}

Set
\begin{equation}\label{eq:Q-minus}
 Q^-=M\cap S^-.
\end{equation}
Since $S^-\subseteq S_0$, we have
\[
 Q^-=M\cap S^-=P\cap S^-.
\]
Again $T\in Q^-$ and $1\notin Q^-$, so this is a nonzero proper prime.

\Needspace{4\baselineskip}
\begin{proposition}\label{prop:negative-prime}
The prime $Q^-$ is nondivisorial in $S^-$.
\end{proposition}

\begin{proof}
The powers $T^e$, $e\in\Q_{>0}$, lie in $Q^-$ by
\cref{eq:T-e-universal,eq:negative-intersection}.  Since
$AQ^-\subseteq A$, one has $A\subseteq(A:Q^-)$.  If $x\notin A$, choose
$e\in\Q_{>0}$ with $e<-v(x)$.  Then $T^e\in Q^-$ but
$xT^e\notin A$, proving $(A:Q^-)=A$.  For every
$i\in\Lambda$ and $m\geq1$, the element $p_{i,m}$ lies in $Q^-$ and is a
unit of $W_{i,m}$ by \cref{eq:p-properties}.  Hence
$(W_{i,m}:Q^-)=W_{i,m}$ for every valuation ring in
\cref{eq:negative-intersection}.  Taking colons through that intersection
gives
\[
 (S^-:Q^-)=S^-.
\]
Consequently
\[
 (Q^-)^{v_{S^-}}=(S^-:(S^-:Q^-))=(S^-:S^-)=S^-\neq Q^-.
\]
Thus $Q^-$ is nondivisorial.
\end{proof}

\begin{proof}[Proof of \cref{thm:main}]
The B\'ezout and quotient-field assertions for the $S_j$ follow from
\cref{lem:intersection} and \cref{eq:H-n,eq:E-n}.  Strictness and proper
containment in $A$ are \cref{prop:strict-chain}; prime compatibility is
\cref{prop:compatibility}; and the localization assertion is
\cref{prop:localization}.  The parity assertion is
\cref{eq:parity-rule}.  Finally, \cref{prop:positive-union,prop:positive-prime}
give every assertion about $S^+$, while
\cref{prop:negative-intersection,prop:negative-prime} give every assertion
about $S^-$.  This proves all parts of the theorem.
\end{proof}

\begin{remark}
The extra coordinate $\ast$ is not cosmetic.  Keeping all
$W_{\ast,m}$ at every stage makes the positive union a proper subring of
$A$ and supplies a unit of each persistent valuation ring for the final
colon calculation.  By contrast, as the chain is followed toward negative
infinity, every integer-indexed family $W_{i,m}$ eventually appears; this is
exactly what produces the explicit intersection in
\cref{eq:negative-intersection}.
\end{remark}

\section{An ascending chain with a divisorial union prime}
\label{sec:divisorial-positive-limit}

We now modify only the nonnegative half of the preceding construction.  The
initial ring and prime will remain $S_0=E_0$ and $P=Q_0$, but at stage $n$
we retain only a tail of the star-indexed valuations.  For $n\geq0$, put
\begin{align}
 C_{\ast,n}&=\bigcap_{m>n}W_{\ast,m},\label{eq:star-tail}\\
 \widetilde H_n
  &=A\cap C_{\ast,n}\cap
       \bigcap_{\substack{i\in\Z,\ i>n\\m\geq1}}W_{i,m},
       \label{eq:H-tilde}\\
 \widetilde E_n
  &=B_n\cap C_{\ast,n}\cap
       \bigcap_{\substack{i\in\Z,\ i>n\\m\geq1}}W_{i,m},
       \label{eq:E-tilde}\\
 \widetilde S_{2n}&=\widetilde E_n,
 \qquad \widetilde S_{2n+1}=\widetilde H_n.
       \label{eq:S-tilde}
\end{align}
Every ring in \cref{eq:S-tilde} is a countable intersection of valuation
rings trivial on $k$, so it is a B\'ezout domain with quotient field $K$ by
\cref{lem:intersection}.  Notice that
\begin{equation}\label{eq:same-initial-ring}
 \widetilde S_0=\widetilde E_0=E_0=S_0,
\end{equation}
because $C_{\ast,0}=C_\ast$.

\begin{proposition}\label{prop:tilde-chain}
For every $n\geq0$,
\[
 \widetilde E_n\subsetneq\widetilde H_n
      \subsetneq\widetilde E_{n+1}.
\]
If $\widetilde Q_j=M\cap\widetilde S_j$, then
\(\widetilde Q_j\) is divisorial in \(\widetilde S_j\) exactly when $j$ is
even.  Moreover,
\[
 P\widetilde S_j=\widetilde Q_j\qquad(j\geq0).
\]
\end{proposition}

\begin{proof}
Since $B_n\subseteq A$, the first inclusion is immediate.  The element
$X_n^{-1}$ has value zero in every valuation occurring in
\cref{eq:H-tilde}, but its $\beta_n$-value is $(0,-1)$.  Hence
\[
 X_n^{-1}\in\widetilde H_n\setminus\widetilde E_n.
\]
The ring $\widetilde H_n$ is contained in $A$ and in the complete family
$(W_{n+1,m})_{m\geq1}$, so \cref{lem:limiting} gives
$\widetilde H_n\subseteq B_{n+1}$.  Passing from
\cref{eq:H-tilde} to \cref{eq:E-tilde} with index $n+1$ replaces the
$A$-condition by the now verified $B_{n+1}$-condition, removes this complete
family, and also removes $W_{\ast,n+1}$; all other conditions are unchanged.
Thus
$\widetilde H_n\subseteq\widetilde E_{n+1}$.  Finally,
\[
 TX_{n+1}^{-2}
   \in\widetilde E_{n+1}\setminus\widetilde H_n.
\]
Indeed, its $v$-value is $1$, its second value at every retained
$W_{\ast,m}$ and every $W_{i,m}$ with $i>n+1$ is $m>0$, whereas its second
value at $W_{n+1,1}$ is $-1$.  This proves both strict inclusions.

We give the divisoriality calculation to verify that deleting finitely many
star conditions changes no finite-stage conclusion.  If $e\in\Q_{>0}$,
then $v(T^e)=e>0$, and at every $W_{i,m}$ one has
$d_{i,m}(T^e)=0$ and
$u_i(\overline{T^e})=u_i(X_i^{me})=me>0$.  Thus
$T^e\in\widetilde Q_j$ for every $j$.  Moreover, for each $W_{i,m}$ in a
defining intersection, \cref{eq:p-properties} gives
$p_{i,m}\in\widetilde Q_j\cap W_{i,m}^{\times}$.

Put $\widetilde Q=\widetilde Q_{2n+1}=M\cap\widetilde H_n$.  We have
\[
 (A:\widetilde Q)=A,
 \qquad
 (W_{i,m}:\widetilde Q)=W_{i,m}.
\]
Indeed, both right-to-left inclusions follow from
$\widetilde Q\subseteq\widetilde H_n\subseteq A\cap W_{i,m}$.  If
$x\notin A$, choose $e\in\Q_{>0}$ with $e<-v(x)$; then
$T^e\in\widetilde Q$ but $xT^e\notin A$.  If $x\notin W_{i,m}$, then
$xp_{i,m}\notin W_{i,m}$ because $p_{i,m}$ is a unit.  These observations
prove the reverse inclusions.  Taking the colon through
\cref{eq:H-tilde} now yields
\begin{equation}\label{eq:tilde-odd-colon}
 (\widetilde H_n:\widetilde Q)=\widetilde H_n,
\end{equation}
so the odd prime is nondivisorial.

For the even stage, put $\widetilde Q=M\cap\widetilde E_n$.  If $a\in A$
and $q\in\widetilde Q$, then $v(aq)>0$, so $aq\in B_n$; hence
$A\subseteq(B_n:\widetilde Q)$.  Conversely, if $x\notin A$, choose
$e\in\Q_{>0}$ with $e<-v(x)$.  Since $T^e\in\widetilde Q$ and
$xT^e\notin A\supseteq B_n$, one has $x\notin(B_n:\widetilde Q)$.
Thus $(B_n:\widetilde Q)=A$.  The unit argument above also gives
$(W_{i,m}:\widetilde Q)=W_{i,m}$ for every retained auxiliary valuation.
Consequently
\begin{equation}\label{eq:tilde-even-first-colon}
 (\widetilde E_n:\widetilde Q)=\widetilde H_n.
\end{equation}
We next compute the second colon.  If
$q\in M\cap\widetilde E_n$ and $h\in\widetilde H_n$, then $q$ and $h$
belong to every common auxiliary valuation ring, while
$v(qh)=v(q)+v(h)>0$.  Thus $qh\in B_n$ and
$q\widetilde H_n\subseteq\widetilde E_n$, proving
$M\cap\widetilde E_n\subseteq(\widetilde E_n:\widetilde H_n)$.

Conversely, suppose $x\widetilde H_n\subseteq\widetilde E_n$.  Because
$1\in\widetilde H_n$, one has $x\in\widetilde E_n$.  If
$x\notin M\cap\widetilde E_n$, then $v(x)=0$.  Choose
$e\in\Q_{>0}$ with $e>u_n(\ac(x))$.  The element $X_n^{-e}$ lies in
$\widetilde H_n$: its $v$-value is zero, and every coordinate valuation
retained in \cref{eq:H-tilde} is different from $u_n$.  But
$u_n(\ac(xX_n^{-e}))=u_n(\ac(x))-e<0$, so
$xX_n^{-e}\notin B_n$ and hence not in $\widetilde E_n$, a contradiction.
We have proved
\begin{equation}\label{eq:tilde-even-second-colon}
 (\widetilde E_n:\widetilde H_n)=M\cap\widetilde E_n.
\end{equation}
Combining \cref{eq:tilde-even-first-colon,eq:tilde-even-second-colon}
makes the even prime divisorial.

It remains only to identify the extensions of $P$.  The strict inclusions
just proved make every $\widetilde S_j$ an overring of
$\widetilde S_0=S_0$.  The contraction of
$\widetilde Q_j=M\cap\widetilde S_j$ to $S_0$ is $M\cap S_0=P$.
Applying \cref{lem:prime-extension} gives
$P\widetilde S_j=\widetilde Q_j$.
\end{proof}

The following tail form of the limiting implication identifies the union.

\begin{lemma}\label{lem:cofinal-tail}
For every $i\in\Lambda$,
\begin{equation}\label{eq:cofinal-tail}
 B_i=\bigcup_{N\geq0}
       \left(A\cap\bigcap_{m>N}W_{i,m}\right).
\end{equation}
\end{lemma}

\begin{proof}
Suppose first that $z$ belongs to one of the rings on the right.  The case
$z=0$ is immediate, so assume $z\neq0$.  If $v(z)>0$, then $z\in B_i$ by
definition.  If $v(z)=0$, eventual regularity and
\cref{eq:asymptotic-value} hold for all sufficiently large $m$.  Choose
such an $m>N$.  Since $z\in W_{i,m}$, its second value is nonnegative, and
\[
 0\leq u_i(z(X_i^m))=u_i(\ac(z)).
\]
Thus $z\in B_i$ in this case also.

Conversely, let $0\neq z\in B_i$.  Equation~\eqref{eq:asymptotic-value}
and the nonvanishing assertion preceding it show
that, for all sufficiently large $m$,
\[
 u_i\bigl(z(X_i^m)\bigr)=m v(z)+u_i(\ac(z)).
\]
If $v(z)>0$, the right side is eventually positive.  If $v(z)=0$, it is
$u_i(\ac(z))\geq0$ by the definition of $B_i$.  Thus $z\in W_{i,m}$ for
all sufficiently large $m$, and $z\in A$ because $B_i\subseteq A$.  The
case $z=0$ is immediate.  This proves the reverse inclusion.
\end{proof}

\begin{proposition}\label{prop:divisorial-positive-limit}
One has
\begin{equation}\label{eq:tilde-positive-union}
 \widetilde S^+:=\bigcup_{j\geq0}\widetilde S_j=B_\ast.
\end{equation}
Moreover, $B_\ast\subsetneq A$.
The induced prime is
\[
 P\widetilde S^+=M,
\]
and it is divisorial in the valuation domain $B_\ast$.
\end{proposition}

\begin{proof}
Every $\widetilde S_j$ is contained in $A$ and in a cofinal tail of the
family $(W_{\ast,m})_{m\geq1}$.  By \cref{lem:cofinal-tail}, it is contained
in $B_\ast$.  Hence $\widetilde S^+\subseteq B_\ast$.

For the converse, take $0\neq z\in B_\ast$.  By
\cref{lem:cofinal-tail}, there is $N$ such that $z\in W_{\ast,m}$ whenever
$m>N$.  Choose a finite set $J\subseteq\Z$ supporting all the coefficients
in a presentation of $z$ as in \cref{eq:finite-presentation}, and then
choose $n\geq N$ larger than every member of $J$.  For $i>n$, the calculation
in \cref{eq:far-coordinate-value} gives
\[
 u_i\bigl(z(X_i^m)\bigr)=m v(z)\geq0
 \qquad(m\geq1).
\]
Thus $z$ satisfies every condition defining $\widetilde H_n$, and hence
$z\in\widetilde S^+$.  Zero belongs to both sides, proving
\cref{eq:tilde-positive-union}.

The containment $B_\ast\subseteq A$ follows from its definition, and it is
strict because $X_\ast^{-1}\in A\setminus B_\ast$.

By definition, $B_\ast$ contains every element of positive $v$-value, so
$M\cap B_\ast=M$.  Its contraction to $S_0$ is $P$; alternatively using
\cref{lem:prime-extension}, we obtain $PB_\ast=M$.

We finally compute both colons.  Since multiplying a nonnegative
$v$-value by a positive one remains positive, $A\subseteq(B_\ast:M)$.  If
$x\notin A$, choose $e\in\Q$ with $0<e<-v(x)$.  Then $T^e\in M$, while
$xT^e\notin B_\ast$.  Therefore
\begin{equation}\label{eq:Bstar-M}
 (B_\ast:M)=A.
\end{equation}
Also $MA\subseteq M$, so $M\subseteq(B_\ast:A)$.  Conversely, if
$xA\subseteq B_\ast$, then $x\in B_\ast$ because $1\in A$.  If $x\notin
M$, then $v(x)=0$.  Choose $e>u_\ast(\ac(x))$.  Although
$X_\ast^{-e}\in A$, one has
\[
 u_\ast\bigl(\ac(xX_\ast^{-e})\bigr)
   =u_\ast(\ac(x))-e<0,
\]
contradicting $xA\subseteq B_\ast$.  Hence
\begin{equation}\label{eq:Bstar-A}
 (B_\ast:A)=M.
\end{equation}
Combining \cref{eq:Bstar-M,eq:Bstar-A} gives
\[
 M^{v_{B_\ast}}=(B_\ast:(B_\ast:M))=(B_\ast:A)=M.
\]
Thus the union prime is divisorial.
\end{proof}

The original chain and the modified chain both begin with $R=S_0$, have
$P=Q_0$, and localize to $R_P=A$ by \cref{prop:localization}.  Their union
primes are nondivisorial and divisorial, respectively, by
\cref{prop:positive-prime,prop:divisorial-positive-limit}.  This proves the
ascending assertion of \cref{thm:limit-alternatives}.

\section{Large-field avoidance and discrete auxiliary data}
\label{sec:large-avoidance}

The descending construction applies \cref{lem:intersection} to controlled
families of cardinality smaller than the ground field.

Choose a characteristic-zero field $k_0$ of cardinality
\begin{equation}\label{eq:singular-cardinal}
 \kappa=\aleph_\omega=\sup_{r<\omega}\aleph_r;
\end{equation}
for example, one may take the rational function field over $\Q$ in
$\kappa$ algebraically independent variables.  Let
\[
 \Gamma_0=\bigoplus_{i\in\Z}\Q e_i,
 \qquad
 F_0=\operatorname{Frac}(k_0[\Gamma_0]),
 \qquad K_0=F_0(t).
\]
As before, $X_i^q$ denotes the group-algebra monomial corresponding to
$qe_i$, and
\[
 u_i:F_0^\times\longrightarrow\Q
\]
is its $i$th coordinate valuation.  Let $a=\operatorname{ord}_t$ on $K_0$,
and write
\begin{equation}\label{eq:A0-M0}
 A_0=\{x:a(x)\geq0\},
 \qquad M_0=\{x:a(x)>0\}=tA_0.
\end{equation}
For $x\neq0$, let $\ac_0(x)$ be its $t$-adic angular component.  Define the
composite valuation ring
\begin{equation}\label{eq:B0i}
 B_i^0=\{x:a(x)>0\}\cup
       \{x:a(x)=0, u_i(\ac_0(x))\geq0\}.
\end{equation}

For $i\in\Z$ and $s\geq1$, let $d_{i,s}$ be the discrete valuation at the
closed point $t=X_i^s$ of the rational function field $F_0(t)$.  Its
residue field is $F_0$.  Refine it by $u_i$ and denote the resulting
rank-two valuation ring by $W^0_{i,s}$.  Thus an element of positive
$d_{i,s}$-value belongs to $W^0_{i,s}$, while an element of zero
$d_{i,s}$-value belongs precisely when the $u_i$-value of its residue is
nonnegative.  All these valuations are trivial on $k_0$.

We need the integer-exponent version of the limiting calculation.  If
$0\neq z=f(t)/g(t)$ and
\[
 f(t)=\sum_{r=r_0}^{r_1}a_rt^r,
 \qquad g(t)=\sum_{r=s_0}^{s_1}b_rt^r,
 \qquad a_{r_0}b_{s_0}\neq0,
\]
then, for all sufficiently large integers $s$, the least displayed term is
the unique term of least $u_i$-value after substituting $t=X_i^s$.  Indeed,
for $r>r_0$ the difference from the value of the $r_0$-term is
\[
 u_i(a_r)-u_i(a_{r_0})+s(r-r_0),
\]
which is eventually positive.  Applying the same estimate to the terms of
$g$ proves the corresponding assertion for its $s_0$-term.  Therefore
the evaluations are nonzero for all sufficiently large $s$ and
\begin{equation}\label{eq:discrete-asymptotic}
 u_i\bigl(z(X_i^s)\bigr)
   =s a(z)+u_i(\ac_0(z)).
\end{equation}
In particular,
\begin{equation}\label{eq:discrete-limiting}
 A_0\cap\bigcap_{s\geq1}W^0_{i,s}\subseteq B_i^0.
\end{equation}

For $n\in\Z$, define
\begin{align}
 H_n^0&=A_0\cap
      \bigcap_{\substack{i>n\\s\geq1}}W^0_{i,s},
      \label{eq:H0n}\\
 E_n^0&=B_n^0\cap
      \bigcap_{\substack{i>n\\s\geq1}}W^0_{i,s},
      \label{eq:E0n}\\
 S_{2n}^0&=E_n^0,
 \qquad S_{2n+1}^0=H_n^0.\label{eq:S0j}
\end{align}
Each defining family is countable, hence has cardinality less than
$\kappa$.  By \cref{lem:intersection}, every $S_j^0$ is a B\'ezout
domain with quotient field $K_0$.

The auxiliary rings satisfy
\begin{equation}\label{eq:discrete-nesting}
 E_n^0\subsetneq H_n^0\subsetneq E_{n+1}^0
 \qquad(n\in\Z).
\end{equation}
Indeed, $B_n^0\subseteq A_0$ gives the first inclusion.  The element
$X_n^{-1}$ lies in every valuation defining $H_n^0$, but its
$u_n$-value is $-1$; hence
\[
 X_n^{-1}\in H_n^0\setminus E_n^0.
\]
By \cref{eq:discrete-limiting}, the complete family
$(W^0_{n+1,s})_{s\geq1}$ in $H_n^0$ forces membership in $B_{n+1}^0$.
The other conditions defining $H_n^0$ are precisely the auxiliary
conditions retained in $E_{n+1}^0$, so
$H_n^0\subseteq E_{n+1}^0$.  This inclusion is strict because
\begin{equation}\label{eq:discrete-strict-witness}
 tX_{n+1}^{-2}\in E_{n+1}^0\setminus H_n^0.
\end{equation}
Its $a$-value is $1$; at $W^0_{n+1,1}$ its first value is zero and its
second value is $-1$, whereas every valuation retained in $E_{n+1}^0$
gives it nonnegative value.  This proves \cref{eq:discrete-nesting}.

We also record the test elements used after geometric filtering.  For
$i\in\Z$ and $s\geq1$, put
\begin{equation}\label{eq:discrete-test-element}
 p^0_{i,s}=tX_i^{-s}
      \prod_{h=1}^{s-1}(1-tX_i^{-h}).
\end{equation}
It has $a$-value $1$.  At $W^0_{i,r}$ with $r<s$, the factor indexed by
$h=r$ has positive first value; at $r=s$, the first value is zero and the
residue has $u_i$-value zero; and at $r>s$, the residue has $u_i$-value
$r-s>0$.  At $W^0_{j,r}$ with $j\neq i$, its residue has $u_j$-value
$r>0$.  Consequently
\begin{equation}\label{eq:discrete-test-properties}
 p^0_{i,s}\in M_0\cap
   \bigcap_{\substack{j\in\Z\\r\geq1}}W^0_{j,r},
 \qquad p^0_{i,s}\in(W^0_{i,s})^\times.
\end{equation}

\section{A descending chain with a divisorial intersection prime}
\label{sec:divisorial-negative-limit}

Let $\Omega$ be the set of closed points of the projective line
$\mathbf P^1_{F_0}$ other than the point $0$ defined by $t=0$.  For
$\xi\in\Omega$, write $\mathcal O_\xi$ for the corresponding discrete
valuation ring of $K_0=F_0(t)$.  We have
\begin{equation}\label{eq:Omega-cardinality}
 |F_0|=|\Omega|=\kappa.
\end{equation}
Indeed, $\Gamma_0$ is countable, so adjoining its group-algebra monomials
and then taking fractions does not increase the infinite cardinality of
$k_0$.  There are at most $\kappa$ irreducible polynomials over $F_0$, while
the $\kappa$ rational points $t=c$, $c\in F_0$, give the reverse inequality.

Well-order $\Omega=(\xi_\alpha)_{\alpha<\kappa}$.  Set
$\mathcal C_0=\varnothing$, and for $m\geq1$ set
\begin{equation}\label{eq:Cm}
 \mathcal C_m=\{\xi_\alpha:\alpha<\aleph_{m-1}\}.
\end{equation}
Then
\begin{equation}\label{eq:Cm-properties}
 \mathcal C_m\subseteq\mathcal C_{m+1},
 \qquad |\mathcal C_m|<\kappa,
 \qquad \bigcup_{m\geq0}\mathcal C_m=\Omega.
\end{equation}
Define
\begin{equation}\label{eq:Dm}
 G_m=\bigcap_{\xi\in\mathcal C_m}\mathcal O_\xi,
 \qquad
 D_m=S^0_{-m}\cap G_m,
 \qquad
 \mathfrak p_m=M_0\cap D_m.
\end{equation}
The valuation family defining $D_m$ has cardinality less than $\kappa$ and
is trivial on $k_0$.  Hence every $D_m$ is a B\'ezout domain with quotient
field $K_0$ by \cref{lem:intersection}.

We shall repeatedly choose a constant $c\in k_0^\times$ whose rational
point $t=c$ is not in a specified $\mathcal C_m$.  Such a choice is always
possible: there are $\kappa$ nonzero constants and fewer than $\kappa$
points in $\mathcal C_m$.  Also, $t=c$ is distinct from every special point
$t=X_i^s$.

\begin{proposition}\label{prop:filtered-chain}
The sequence
\[
 D_0\supsetneq D_1\supsetneq D_2\supsetneq\cdots
\]
is a strictly descending chain of B\'ezout domains with common quotient
field $K_0$.  Moreover,
\begin{equation}\label{eq:filtered-parity}
 \mathfrak p_m\text{ is divisorial in }D_m
 \quad\Longleftrightarrow\quad m\text{ is even}.
\end{equation}
\end{proposition}

\begin{proof}
The inclusions follow from \cref{eq:discrete-nesting}, which gives
$S^0_{-(m+1)}\subseteq S^0_{-m}$, and from
$G_{m+1}\subseteq G_m$.  We first prove they are strict.  Suppose $m$ is
even and write $-m=2n$.  Choose $c\in k_0^\times$ with
$t=c\notin\mathcal C_{m+1}$, and set
\begin{equation}\label{eq:filtered-even-witness}
 y=\frac{tX_n^{-2}}{t-c}.
\end{equation}
The numerator is the witness in \cref{eq:discrete-strict-witness} for
$E_n^0\setminus H_{n-1}^0$.  The denominator has value zero, with
second value zero, at every special valuation $W^0_{i,s}$.  Consequently
$y\in E_n^0=S^0_{-m}$, but $y\notin H_{n-1}^0=S^0_{-(m+1)}$; explicitly,
its second value at $W^0_{n,1}$ is $-1$.  As a rational function of $t$,
$y$ has its only pole at $t=c$ and is regular at infinity because its
numerator and denominator have the same degree.  Hence
$y\in G_{m+1}\subseteq G_m$.  It follows that
$y\in D_m\setminus D_{m+1}$.

If $m$ is odd, write $-m=2n+1$.  The constant rational function
$X_n^{-1}$ has no geometric pole and belongs to
$H_n^0\setminus E_n^0$.  It therefore lies in
$D_m\setminus D_{m+1}$.  This proves strictness in both cases.

Fix $m$ and choose $c\in k_0^\times$ with
$t=c\notin\mathcal C_m$.  The element
\begin{equation}\label{eq:geometric-unit}
 q_c=\frac{t}{t-c}
\end{equation}
has $a$-value $1$.  At $W^0_{i,s}$ its first value is zero and its second
value is
\[
 u_i\left(\frac{X_i^s}{X_i^s-c}\right)=s.
\]
Its divisor on $\mathbf P^1_{F_0}$ consists of a zero at $0$ and a pole at
$c$, and it is a unit at infinity.  Thus
\begin{equation}\label{eq:qc-properties}
 q_c\in\mathfrak p_m,
 \qquad q_c\in\mathcal O_\xi^\times
       \quad(\xi\in\mathcal C_m).
\end{equation}
In particular, every $\mathfrak p_m$ is nonzero.

For $i\in\Z$ and $s\geq1$, regularize the element from
\cref{eq:discrete-test-element} by putting
\begin{equation}\label{eq:regularized-test}
 r_{i,s,c}=\frac{p^0_{i,s}}{(t-c)^s}.
\end{equation}
The polynomial $p^0_{i,s}$ has degree $s$, so $r_{i,s,c}$ is regular at
infinity and its only possible pole is $t=c$.  It therefore belongs to
$G_m$.  The denominator $t-c$ is a unit, of second value zero, at every
$W^0_{j,h}$, so \cref{eq:discrete-test-properties} gives
\begin{equation}\label{eq:regularized-properties}
 r_{i,s,c}\in M_0\cap D_m
 \quad\text{whenever }W^0_{i,s}\text{ occurs in }S^0_{-m},
 \qquad
 r_{i,s,c}\in(W^0_{i,s})^\times.
\end{equation}
The first assertion follows because $p^0_{i,s}$ belongs to every special
valuation, and its positive $a$-value puts it in either $A_0$ or $B_n^0$.

Suppose first that $m$ is odd, so $S^0_{-m}=H_n^0$ for some $n$.  Put
$\mathfrak p=\mathfrak p_m$.  Since all elements of $\mathfrak p$ have
positive integral $a$-value, while $q_c\in\mathfrak p$ has value $1$,
\begin{equation}\label{eq:filtered-A-colon}
 (A_0:\mathfrak p)=t^{-1}A_0.
\end{equation}
For every special valuation in $H_n^0$, the appropriate element
$r_{i,s,c}$ is a $\mathfrak p$-unit, so
$(W^0_{i,s}:\mathfrak p)=W^0_{i,s}$.  For every
$\xi\in\mathcal C_m$, the element $q_c$ is a $\mathfrak p$-unit, so
$(\mathcal O_\xi:\mathfrak p)=\mathcal O_\xi$.  Taking the colon through
the defining intersection for $D_m$ gives
\[
 (D_m:\mathfrak p)=t^{-1}A_0\cap
   \bigcap_{\substack{i>n\\s\geq1}}W^0_{i,s}
   \cap G_m.
\]
Every element of the displayed intersection has $a$-value at least $-1$.
Suppose that $x$ in this intersection has $a(x)=-1$, and fix any integer
$i>n$.  For all sufficiently large $s$, \cref{eq:discrete-asymptotic}
applies and gives
\[
 u_i\bigl(x(X_i^s)\bigr)=-s+u_i(\ac_0(x))<0.
\]
At such an $s$, the element $x$ is regular and nonzero at $t=X_i^s$, so
this strict inequality says $x\notin W^0_{i,s}$, contradicting the
displayed intersection.  Thus every element in it has nonnegative
$a$-value.  Its remaining special and geometric conditions are exactly
those defining $D_m$, so the displayed intersection is $D_m$.  Therefore
\[
 \mathfrak p^{v_{D_m}}=(D_m:(D_m:\mathfrak p))=D_m\neq\mathfrak p,
\]
and the odd prime is nondivisorial.

Suppose now that $m$ is even, so $S^0_{-m}=E_n^0$.  Again write
$\mathfrak p=\mathfrak p_m$.  We claim
\begin{equation}\label{eq:filtered-B-colon}
 (B_n^0:\mathfrak p)=A_0.
\end{equation}
If $x\in A_0$ and $q\in\mathfrak p$, then
$a(xq)=a(x)+a(q)>0$, so $xq\in B_n^0$.  This proves
$A_0\subseteq(B_n^0:\mathfrak p)$.  Since the $t$-adic valuation $a$ is
integer-valued, an element outside $A_0$ has value at most $-1$.  If
$a(x)\leq-2$, then $xq_c\notin B_n^0$.  If $a(x)=-1$, choose an integer
$s>u_n(\ac_0(x))$.  The regularized element $r_{n,s,c}$ belongs to
$\mathfrak p$: although $W^0_{n,s}$ is not one of the conditions defining
$E_n^0$, \cref{eq:discrete-test-properties} says that $p^0_{n,s}$ belongs
to every retained special valuation, and the denominator is a unit there.
At $t=0$ its angular component differs from $X_n^{-s}$ by a
$u_n$-unit.  Consequently
\[
 a(xr_{n,s,c})=0,
 \qquad
 u_n\bigl(\ac_0(xr_{n,s,c})\bigr)
   =u_n(\ac_0(x))-s<0,
\]
so $xr_{n,s,c}\notin B_n^0$.  This proves
\cref{eq:filtered-B-colon}.

The units in \cref{eq:qc-properties,eq:regularized-properties} now give
\begin{equation}\label{eq:filtered-even-first-colon}
 (D_m:\mathfrak p)=L_m,
 \qquad
 L_m:=H_n^0\cap G_m.
\end{equation}
We also have
\begin{equation}\label{eq:filtered-even-conductor}
 (D_m:L_m)=\mathfrak p.
\end{equation}
Indeed, $\mathfrak pL_m\subseteq D_m$: all auxiliary and geometric
conditions are shared, and the product has positive $a$-value, so it lies
in $B_n^0$.  Conversely, suppose $xL_m\subseteq D_m$.  Since $1\in L_m$,
$x\in D_m$.  If $x\notin\mathfrak p$, then $a(x)=0$ and
$u_n(\ac_0(x))\geq0$.  Choose $e\in\Q$ larger than the latter value.  The
constant $X_n^{-e}$ has no geometric pole and lies in every special
valuation defining $H_n^0$, so $X_n^{-e}\in L_m$.  But
$xX_n^{-e}\notin B_n^0$, contradicting $xL_m\subseteq D_m$.  This proves
\cref{eq:filtered-even-conductor}, and hence
\[
 \mathfrak p^{v_{D_m}}
  =(D_m:(D_m:\mathfrak p))=(D_m:L_m)=\mathfrak p.
\]
The parity assertion follows because $m$ and $-m$ have the same parity.
\end{proof}

Set
\begin{equation}\label{eq:E-infinity}
 D_\infty=\bigcap_{m\geq0}D_m.
\end{equation}
The geometric filtration was chosen so that the prime disappears at this
intermediate limit.

\begin{lemma}\label{lem:base-prime-vanishes}
One has $D_\infty\subseteq F_0$ and
\begin{equation}\label{eq:base-prime-zero}
 \bigcap_{m\geq0}\mathfrak p_m=M_0\cap D_\infty=0.
\end{equation}
\end{lemma}

\begin{proof}
Every $D_m$ is contained in $A_0$, so an element of $D_\infty$ is regular
at the point $0$.  Given any other closed point $\xi$, choose $m$ with
$\xi\in\mathcal C_m$ by \cref{eq:Cm-properties}.  Then
$D_\infty\subseteq D_m\subseteq\mathcal O_\xi$.  Hence every element of
$D_\infty$ is regular at every closed point of $\mathbf P^1_{F_0}$.

For completeness, a rational function regular everywhere is constant.
Write it as $f(t)/g(t)$ with coprime polynomials in $F_0[t]$.  If $g$ is
nonconstant, any irreducible factor of $g$ gives a finite pole.  Thus $g$
is constant and the function is a polynomial.  Regularity at infinity then
forces its degree to be zero.  This proves $D_\infty\subseteq F_0$.
Every nonzero member of $F_0$ has $a$-value zero, whereas the nonzero
members of $M_0$ have positive $a$-value.  Their intersection is therefore
zero, proving \cref{eq:base-prime-zero}.
\end{proof}

We now turn the zero in \cref{eq:base-prime-zero} into a nonzero
divisorial prime without changing the finite-stage parity.  The following
pullback calculation contains every fact needed for that step.

\begin{lemma}[Nondiscrete valuation pullback]\label{lem:outer-pullback}
Let $K$ be a field, let $V$ be a rank-one valuation domain with quotient
field $L$, nondiscrete value group, maximal ideal $N$, residue field $K$,
and residue map $\varphi:V\to K$.  Suppose that $D$ is a Pr\"ufer subdomain
of $K$ with quotient field $K$, and put
\[
 R=\varphi^{-1}(D).
\]
Then $R$ is a Pr\"ufer domain with quotient field $L$; if $D$ is B\'ezout,
then $R$ is B\'ezout.  If
$0\neq\mathfrak a$ is an ideal of $D$ and
$I=\varphi^{-1}(\mathfrak a)$, then
\begin{equation}\label{eq:pullback-v-transfer}
 I^{v_R}=\varphi^{-1}(\mathfrak a^{v_D}).
\end{equation}
\end{lemma}

\begin{proof}
Write $w$ for the valuation belonging to $V$.
We first prove the Pr\"ufer assertion by examining the localizations of
$R$.  A maximal ideal of $R$ which contains $N$ has the form
$\mathfrak q=\varphi^{-1}(\mathfrak m)$ for a maximal ideal
$\mathfrak m$ of $D$.  We claim
\begin{equation}\label{eq:pullback-localization}
 R_{\mathfrak q}=\varphi^{-1}(D_{\mathfrak m}).
\end{equation}
The inclusion from left to right is clear.  Conversely, let $x$ belong to
the right side.  If its residue is $d/s$ with $d\in D$ and
$s\in D\setminus\mathfrak m$, choose a lift $\widetilde s\in R$ with
residue $s$.  Then $\widetilde s x$ has residue $d$, so
$\widetilde s x\in R$, while $\widetilde s\notin\mathfrak q$.  Hence
$x\in R_{\mathfrak q}$.  Since
$D_{\mathfrak m}$ is a valuation domain, its inverse image on the right of
\cref{eq:pullback-localization} is the composite valuation domain: for an
element of $L^\times$ of outer value zero, membership is decided by the
residue valuation, and elements of positive or negative outer value are
decided by $V$.

If a maximal ideal $\mathfrak q$ of $R$ does not contain $N$, choose
$y\in N\setminus\mathfrak q$.  For every $x\in V$, one has
$yx\in N\subseteq R$, and hence $x=(yx)/y\in R_{\mathfrak q}$.  Thus
$R_{\mathfrak q}$ is an overring of the valuation domain $V$ and is itself
a valuation domain.  Every localization of $R$ at a maximal ideal is
therefore a valuation domain, so $R$ is Pr\"ufer.

Choose $0\neq y\in N$.  For every $x\in V$, both $y$ and $yx$ lie in $R$,
so $x\in\qf(R)$.  It follows that $V\subseteq\qf(R)$ and hence
$\qf(R)=\qf(V)=L$.

Suppose now that $D$ is B\'ezout, and take $0\neq x,y\in R$.  If
$w(x)<w(y)$, then $y/x\in N\subseteq R$, so $xR+yR=xR$; the case
$w(y)<w(x)$ is symmetric.  Assume $w(x)=w(y)=\alpha$.  Choose
$z\in L^\times$ with $w(z)=\alpha$ and set
\[
 a=\varphi(x/z),\qquad b=\varphi(y/z)\in K^\times.
\]
When $\alpha=0$ we take $z=1$, and then $a,b\in D$.  In all cases the
two-generated fractional ideal $aD+bD$ is principal.  Indeed, choose
$0\neq h\in D$ with $ha,hb\in D$; the integral ideal
$haD+hbD$ is $\eta D$ for some $\eta\in D$, and hence
$aD+bD=(\eta/h)D$.  Write
$aD+bD=\delta D$ with $\delta\in K^\times$.  Since
$\delta\in aD+bD$, choose $r,s\in D$ such that $\delta=ra+sb$, choose
lifts $\widetilde r,\widetilde s\in R$ of $r,s$, and put
$d=\widetilde r x+\widetilde s y$.  The residue of $d/z$
is the nonzero element $\delta$, so $w(d)=\alpha$.  Moreover, $x/d$ and
$y/d$ have outer value zero and residues $a/\delta,b/\delta\in D$.
Therefore $x/d,y/d\in R$, while $d\in xR+yR$, and hence
$xR+yR=dR$.  Induction shows that every finitely generated ideal of $R$ is
principal.  Thus $R$ is B\'ezout.

It remains to compute the two colons.  Because the value group is
nondiscrete, it has positive elements arbitrarily close to zero.  If
$x\notin V$, choose $y\in N$ with
$0<w(y)<-w(x)$.  Since
$N\subseteq I$, the element $y$ belongs to $I$, but $xy\notin V$ and hence
$xy\notin R$.  Thus $(R:I)\subseteq V$.  For $x\in V$, reducing products
modulo $N$ gives
\[
 xI\subseteq R
 \quad\Longleftrightarrow\quad
 \varphi(x)\mathfrak a\subseteq D.
\]
Here the implication also covers $x\in N$, whose residue is zero.  Therefore
\begin{equation}\label{eq:pullback-first-colon}
 (R:I)=\varphi^{-1}((D:\mathfrak a)).
\end{equation}
Put $J=(R:I)=\varphi^{-1}((D:\mathfrak a))$; in particular,
$N\subseteq J$.  If $x\notin V$, nondiscreteness again supplies
$y\in N$ with $0<w(y)<-w(x)$.  Then $y\in J$ but $xy\notin V$, so
$x\notin(R:J)$.  Hence $(R:J)\subseteq V$.  For $x\in V$, reduction
modulo $N$ gives
\[
 xJ\subseteq R
 \quad\Longleftrightarrow\quad
 \varphi(x)(D:\mathfrak a)\subseteq D.
\]
Consequently
\[
 (R:(R:I))=(R:J)
  =\varphi^{-1}((D:(D:\mathfrak a)))
  =\varphi^{-1}(\mathfrak a^{v_D}),
\]
which is \cref{eq:pullback-v-transfer}.
\end{proof}

\medskip
Apply the lemma explicitly as follows.  Let
\[
 L=\operatorname{Frac}(K_0[Z^{\Q}])
\]
and let $w$ be the $Z$-exponent valuation.  Its valuation ring $V$, maximal
ideal $N=\{x:w(x)>0\}$, and residue map
\[
 \varphi:V\longrightarrow K_0
\]
have value group $\Q$ and residue field $K_0$.  In particular, $V$ is
rank one and nondiscrete.  Define
\begin{equation}\label{eq:outer-chain}
 R_m=\varphi^{-1}(D_m),
 \qquad
 P=\varphi^{-1}(\mathfrak p_0).
\end{equation}

\begin{proposition}\label{prop:divisorial-negative-limit}
The rings in \cref{eq:outer-chain} form a strictly descending chain
\[
 R_0\supsetneq R_1\supsetneq R_2\supsetneq\cdots
\]
of B\'ezout domains with common quotient field $L$.  The ideal $P$ is a
nonzero prime of $R_0$, and
\begin{equation}\label{eq:outer-parity}
 P\cap R_m\text{ is divisorial in }R_m
 \quad\Longleftrightarrow\quad m\text{ is even}.
\end{equation}
If
\[
 R_\infty=\bigcap_{m\geq0}R_m,
\]
then $P\cap R_\infty=N$, and this is a nonzero divisorial prime of
$R_\infty$.
\end{proposition}

\begin{proof}
Each $D_m$ is B\'ezout with quotient field $K_0$, so
\cref{lem:outer-pullback} makes $R_m$ B\'ezout (and hence Pr\"ufer) with
quotient field $L$.
Taking inverse images preserves the inclusions, and strictness is preserved
as well: if $d\in D_m\setminus D_{m+1}$, its constant lift in $V$ belongs
to $R_m\setminus R_{m+1}$.  Thus the chain is strictly descending.

The ideal $\mathfrak p_0=M_0\cap D_0$ is a nonzero prime, so its inverse
image $P$ is a nonzero prime of $R_0$.  Since $D_m\subseteq D_0$,
\begin{align}
 P\cap R_m
 &=\varphi^{-1}(\mathfrak p_0\cap D_m)\notag\\
 &=\varphi^{-1}(M_0\cap D_m)
  =\varphi^{-1}(\mathfrak p_m).
  \label{eq:outer-contraction}
\end{align}
Every $\mathfrak p_m$ is nonzero by \cref{eq:qc-properties}.  Therefore
\cref{lem:outer-pullback,eq:filtered-parity} give
\[
 (P\cap R_m)^{v_{R_m}}
  =\varphi^{-1}(\mathfrak p_m^{v_{D_m}}),
\]
and this equals $P\cap R_m$ exactly for even $m$.  This proves
\cref{eq:outer-parity}.

Inverse images commute with arbitrary intersections, so
\begin{equation}\label{eq:outer-limit-ring}
 R_\infty=\varphi^{-1}(D_\infty).
\end{equation}
Combining \cref{eq:base-prime-zero,eq:outer-contraction} gives
\begin{equation}\label{eq:outer-limit-prime}
 P\cap R_\infty
 =\varphi^{-1}(\mathfrak p_0\cap D_\infty)
 =\varphi^{-1}(0)=N.
\end{equation}
The ring $R_\infty$ has quotient field $L$: if $0\neq y\in N$, then for
every $x\in V$, both $y$ and $yx$ lie in $N\subseteq R_\infty$, so
$V\subseteq\qf(R_\infty)$ and hence $L=\qf(V)\subseteq\qf(R_\infty)$.
Conversely, $\qf(R_\infty)\subseteq L$ because
$R_\infty\subseteq V\subseteq L$.  Thus $\qf(R_\infty)=L$.

We finish by computing the two colons in this quotient field.  Since $N$ is
an ideal of $V$,
\[
 V\subseteq(R_\infty:N).
\]
If $x\notin V$, choose $e\in\Q$ with $0<e<-w(x)$.  Then
$Z^e\in N$, but $xZ^e\notin V$ and hence $xZ^e\notin R_\infty$.  Thus
\begin{equation}\label{eq:limit-N-first-colon}
 (R_\infty:N)=V.
\end{equation}
Also $NV=N\subseteq R_\infty$, so
$N\subseteq(R_\infty:V)$.  Conversely, suppose $xV\subseteq R_\infty$.
Since $1\in V$, one has $x\in R_\infty\subseteq V$.  If $x\notin N$, its
residue is a nonzero element $b\in D_\infty\subseteq F_0$ by
\cref{lem:base-prime-vanishes}.  The constant lift of the element
$t\in K_0$ is a unit of $V$, but the residue of $xt$ is $bt\notin F_0$;
in particular, it is not in $D_\infty$.  Hence $xt\notin R_\infty$, a
contradiction.  We have proved
\begin{equation}\label{eq:limit-N-second-colon}
 (R_\infty:V)=N.
\end{equation}
Combining \cref{eq:limit-N-first-colon,eq:limit-N-second-colon},
\[
 N^{v_{R_\infty}}
  =(R_\infty:(R_\infty:N))=(R_\infty:V)=N.
\]
Thus the limit prime is nonzero and divisorial.
\end{proof}

The chain in \cref{prop:divisorial-negative-limit} satisfies every
hypothesis of the descending question and has a divisorial intersection
prime.  The reversed nonpositive half of the original bilateral chain
satisfies the same hypotheses and has a nondivisorial intersection prime by
\cref{thm:main}.  Together with the
conclusion of \cref{sec:divisorial-positive-limit}, this proves
\cref{thm:limit-alternatives}.

\renewcommand{\sectionname}{}
\section*{Acknowledgments}
We used GPT-5.6 Sol and an agentic harness built around GPT-5.6 Sol and Claude Fable 5 to assist with literature searches, hypothesis testing, the exploration and elimination of potential approaches, wording refinement, and manuscript proofreading. We thank Hieu M. Vu, Tho Tran Huu, Khoi M. N. Nguyen, Dung V. Nguyen, and Quang X. Nguyen for their assistance with hardware-related matters and for providing technical support in the use of the AI tools and agentic system. Viet-Hoang Tran thanks Thieu N. Vo for insightful discussions regarding the problem.

\bibliographystyle{plain}
\begingroup
\footnotesize
\raggedright
\bibliography{references}

@incollection{CFFG2014,
  author    = {Cahen, Paul-Jean and Fontana, Marco and Frisch, Sophie and Glaz, Sarah},
  title     = {Open Problems in Commutative Ring Theory},
  booktitle = {Commutative Algebra: Recent Advances in Commutative Rings, Integer-Valued Polynomials, and Polynomial Functions},
  editor    = {Fontana, Marco and Frisch, Sophie and Glaz, Sarah},
  pages     = {353--375},
  publisher = {Springer},
  address   = {New York},
  year      = {2014},
  doi       = {10.1007/978-1-4939-0925-4_20}
}

@article{FHP1984,
  author  = {Fontana, Marco and Huckaba, James A. and Papick, Ira J.},
  title   = {Divisorial Prime Ideals in {Pr\"ufer} Domains},
  journal = {Canadian Mathematical Bulletin},
  volume  = {27},
  number  = {3},
  pages   = {324--328},
  year    = {1984},
  doi     = {10.4153/CMB-1984-049-9}
}

@article{FHP1986,
  author  = {Fontana, Marco and Huckaba, James A. and Papick, Ira J.},
  title   = {Some Properties of Divisorial Prime Ideals in {Pr\"ufer} Domains},
  journal = {Journal of Pure and Applied Algebra},
  volume  = {39},
  number  = {1--2},
  pages   = {95--103},
  year    = {1986},
  doi     = {10.1016/0022-4049(86)90137-4}
}

@article{FHL2010,
  author  = {Fontana, Marco and Houston, Evan and Lucas, Thomas G.},
  title   = {Toward a Classification of Prime Ideals in {Pr\"ufer} Domains},
  journal = {Forum Mathematicum},
  volume  = {22},
  number  = {4},
  pages   = {741--766},
  year    = {2010},
  doi     = {10.1515/forum.2010.041}
}

@article{HP1982,
  author  = {Huckaba, James A. and Papick, Ira J.},
  title   = {When the Dual of an Ideal Is a Ring},
  journal = {Manuscripta Mathematica},
  volume  = {37},
  pages   = {67--85},
  year    = {1982},
  doi     = {10.1007/BF01239947}
}

@article{Lucas2015,
  author  = {Lucas, Thomas G.},
  title   = {Divisorial Behavior of Prime Ideals in {Pr\"ufer} Domains},
  journal = {Rendiconti del Circolo Matematico di Palermo. Series II},
  volume  = {64},
  number  = {2},
  pages   = {277--290},
  year    = {2015},
  doi     = {10.1007/s12215-015-0199-5}
}

@incollection{Lucas2017,
  author    = {Lucas, Thomas G.},
  title     = {Divisorial Prime Ideals in {Pr\"ufer} Domains},
  booktitle = {Rings, Polynomials, and Modules},
  editor    = {Fontana, Marco and Frisch, Sophie and Glaz, Sarah and Tartarone, Francesca and Zanardo, Paolo},
  pages     = {281--298},
  publisher = {Springer},
  address   = {Cham},
  year      = {2017},
  doi       = {10.1007/978-3-319-65874-2_14}
}
\endgroup

\end{document}